\documentclass{amsart}

\usepackage{url,amssymb,enumerate,colonequals}

\usepackage{mathrsfs} 
\usepackage[section]{placeins}
\usepackage{MnSymbol}
\usepackage{extarrows}
\usepackage{lscape}
\usepackage[all,cmtip]{xy}

\usepackage[OT2,T1]{fontenc}
\usepackage{color}

\usepackage[
        colorlinks, citecolor=darkgreen,
        backref,
        pdfauthor={Nuno Freitas},
]{hyperref}

\usepackage{comment}
\usepackage{multirow}

\newcommand{\F}{\mathbb{F}}

\newcommand{\Q}{\mathbb{Q}}

\newcommand{\Z}{\mathbb{Z}}

\newcommand{\cF}{\mathcal{F}}

\newcommand{\fp}{\mathfrak{p}}

\newcommand{\OO}{\mathcal{O}}

\DeclareMathOperator{\Gal}{Gal}

\DeclareMathOperator{\Norm}{Norm}

\DeclareMathOperator{\ord}{ord}

\DeclareMathOperator{\Res}{Res}

\numberwithin{equation}{section}

\newtheorem{theorem}{Theorem}

\newtheorem{lemma}{Lemma}

\newtheorem{corollary}{Corollary}

\theoremstyle{definition}

\theoremstyle{remark}

\definecolor{darkgreen}{rgb}{0,0.5,0}

\DeclareRobustCommand{\SkipTocEntry}[5]{}

\begin{document}

\title{
On the Unit equation over cyclic number fields of prime degree
}

\begin{abstract}
Let $\ell \ne 3$ be a prime.
We show that there are only finitely many cyclic
number fields $F$ of degree $\ell$ for which the unit equation
$$\lambda + \mu = 1, \qquad \lambda, \mu \in \OO_F^\times$$
has solutions. Our result is effective. For example,
we deduce that the only cyclic quintic number field
for which the unit equation has solutions is
$\Q(\zeta_{11})^+$.
\end{abstract}

\author{Nuno Freitas}

\address{Departament de Matem\`atiques i Inform\`atica,
Universitat de Barcelona (UB),
Gran Via de les Corts Catalanes 585,
08007 Barcelona, Spain}

\email{nunobfreitas@gmail.com}

\author{Alain Kraus}
\address{Sorbonne Universit\'e,
Institut de Math\'ematiques de Jussieu - Paris Rive Gauche,
UMR 7586 CNRS - Paris Diderot,
4 Place Jussieu, 75005 Paris,
France}
\email{alain.kraus@imj-prg.fr}

\author{Samir Siksek}

\address{Mathematics Institute\\
    University of Warwick\\
    CV4 7AL \\
    United Kingdom}

\email{s.siksek@warwick.ac.uk}

\date{\today}
\thanks{Freitas is supported by a Ram\'on y Cajal fellowship with reference RYC-2017-22262.
Siksek is supported by the
EPSRC grant \emph{Moduli of Elliptic curves and Classical Diophantine Problems}
(EP/S031537/1).}
\keywords{Unit equation, cyclic fields, exceptional fields}
\subjclass[2010]{Primary 11R20}

\maketitle

\section{Introduction}
Let $F$ be a number field. Write $\OO_F$ for the integers of $F$,
and $\OO_F^\times$ for the unit group of $\OO_F$. 
A famous theorem of Siegel \cite{Siegel} asserts that
the  
\textbf{unit equation},
\begin{equation}\label{eqn:unit}
 \lambda + \mu = 1, \quad \lambda, \; \mu \in \OO_F^\times.
\end{equation}
has finitely many solutions.
Unit equations have been the subject
of research for over a century. 
Effective bounds
for the number and heights of the
solutions have
been supplied by many authors \cite[Chapter 4]{EG}.
One of the most elegant such results
is due to Evertse \cite{Evertse}, and asserts
that \eqref{eqn:unit}
has at most $3 \times 7^{3 r+4s}$ solutions, where 
$(r,s)$ is the signature of $F$. The latest effective
bounds on the heights of solutions are due to Gy\H{o}ry \cite{Gyory}.
Moreover, de Weger \cite{dW} 
has given a rather efficient algorithm for determining
the solutions to \eqref{eqn:unit} which combines
Baker's bounds for linear forms in logarithms with
the LLL algorithm.  De Weger's algorithm 
has
since been refined by a number of authors,
for example \cite{Adm}, \cite{RM}, \cite{Smart}.

It is natural to consider the existence of solutions
to \eqref{eqn:unit}. 
Nagell \cite{Nagell} 
calls a unit $\lambda \in
\OO_F^\times$ \textbf{exceptional} if $1-\lambda \in \OO_F^\times$.
The number field $F$ is called \textbf{exceptional} if it possesses
an exceptional unit. Thus $\lambda$ is exceptional
if and only if $(\lambda,1-\lambda)$ is a solution to the
unit equation \eqref{eqn:unit}, and $F$ is exceptional if and only if the unit
equation has solutions. 
In a series of papers spanning over 40 years, starting with \cite{Nagell1928}
and culminating
in \cite{Nagell2},
Nagell determined all exceptional number fields
where the unit rank is $0$ or $1$. For example, Nagell finds
\cite[Section 2]{Nagell2}
that the only exceptional quadratic fields are $\Q(\sqrt{5})$
and $\Q(\sqrt{-3})$ which contain
exceptional units $(1+\sqrt{5})/2$ and $(1-\sqrt{-3})/2$ respectively, 
and the only exceptional complex cubic
fields are the ones with discriminants $-23$ and $-31$.
He also showed \cite[Sections 3--5]{Nagell2} 
that the only exceptional real cubic fields (whence the 
unit rank is $2$)
are of the form $\Q(\lambda)$ where $\lambda$ is a root of
\[
f_k(X)=X^3+(k-1)X^2-kX-1, \qquad k \in \Z, \quad k \ge 3
\]
or of
\[
g_k(X)=X^3+kX^2-(k+3)X+1, \qquad k \in \Z, \quad k \ge -1;
\]
in both cases $\lambda$ is an exceptional unit. It turns out the 
fields $\Q(\lambda)$ defined by the $f_k(X)$ are non-Galois,
whereas the ones defined by the $g_k(X)$ are cyclic (and so Galois),
having discriminant $(k^2+3k+9)^2$.
By a \textbf{cyclic} number field we mean a finite Galois extension of $\Q$
whose Galois group is cyclic.

An interesting problem is determining whether
a family of number fields has exceptional
members.
Beyond the work of Nagell, there are relatively few
works on this problem.
A  beautiful example of such a result
is due to Triantafillou \cite{Triantafillou}:
if $3$ totally splits in a number field $F$ 
and $3 \nmid [F:\Q]$ then $F$ is non-exceptional.
Another example of such a result is found in 
\cite{FKS2}:
if $F$ is a Galois $p$-extension,
where $p \ge 5$ is a prime that totally ramifies
in $F$, then $F$ is non-exceptional.

In this note we consider the problem of determining
exceptional number fields that are cyclic of prime degree.
\begin{theorem}\label{thm:finiteness}
Let $\ell \ne 3$ be a prime. Then there are only finitely many
cyclic number fields $F$ of degree $\ell$ such that $F$
is exceptional.
\end{theorem}
For $\ell=2$ the theorem is due to Nagell \cite{Nagell2}
who showed, as observed above, that the only exceptional
quadratic fields are $\Q(\sqrt{5})$ and $\Q(\sqrt{-3})$.
For $\ell=3$ the theorem is false. Indeed, as already observed
the fields defined by the polynomials $g_k$ are cyclic cubic and exceptional,
and Nagell showed \cite[Th\'eor\`eme 7]{Nagell2} that this family contains infinitely many pairwise non-isomorphic members.
For $\ell \ge 5$,
Theorem~\ref{thm:finiteness} is an immediate 
consequence of the following more 
precise theorem.
\begin{theorem}\label{thm:discs}
Let $\ell \ge 5$ be a prime,
and write
\begin{equation}\label{eqn:Rl}
R_\ell = \Res(X^{2\ell}-1,(X-1)^{2\ell}-1),
\end{equation}
where $\Res$ denotes the resultant.
Then $R_\ell \ne 0$.
Let
\begin{equation}\label{eqn:Sl}
S_\ell=\{ p \mid R_\ell \; : \; \text{$p$ is a prime $\equiv 1 \bmod{\ell}$}\}.
\end{equation}
Let $F$ be a cyclic number field of degree $\ell$, and suppose
the unit equation \eqref{eqn:unit} has solutions.
Write $\Delta_F$ for the discriminant of $\OO_F$,
and $N_F$ for the conductor of $F$.
Then there is a non-empty subset $T \subseteq S_\ell$ such that
\begin{equation}\label{eqn:disc}
\Delta_F= \prod_{p \in T} p^{\ell-1}, \qquad
N_F=\prod_{p \in T} p.
\end{equation}
\end{theorem}
We recall that the \textbf{conductor} of a finite abelian extension $F/\Q$
is the smallest $n$ such that $F \subseteq \Q(\zeta_n)$, where $\zeta_n=\exp(2\pi i/n)$.
Theorem~\ref{thm:discs} is effective, in the sense that 
given a prime $\ell \ge 5$, it gives
an effective algorithm for determine all exceptional cyclic number fields
of degree $\ell$. Indeed, the theorem yields a finite list
of cyclic fields of degree $\ell$ that maybe exceptional,
and for each such cyclic field  
we can simply solve the unit equation 
using de Weger's aforementioned algorithm to decide if it exceptional
or not. We illustrate this by establishing
the following corollary.
\begin{corollary}\label{cor}
The only exceptional cyclic quintic field 
is $F=\Q(\zeta_{11})^+$.
\end{corollary}
The proof of Corollary~\ref{cor} is found in Section~\ref{sec:cor}.

\bigskip

\noindent \textbf{Remark.}
Let $\cF$ be the collection of all exceptional cyclic fields
of prime degree $\ne 3$. 
It is natural in view of the above
results to wonder if $\cF$ is finite or infinite.
We believe that $\cF$ is infinite, as we now explain.
First let $p \ge 5$ be a prime, and let $F=\Q(\zeta_p)^+$.
We will show that 
$F$ is exceptional by exhibiting a solution to the unit
equation \eqref{eqn:unit}. Let $\lambda=2+\zeta_p+\zeta_p^{-1}$
and $\mu=-1-\zeta_p-\zeta_p^{-1}$. Then $\lambda$, $\mu$
belong to $\OO_F$ and satisfy $\lambda+\mu=1$.
We need to show that $\lambda$, $\mu$ units in $\OO_F$
and for this it is in fact enough to show that they
are units in $\Z[\zeta_p]$. Recall that the unique prime
ideal above $p$ in $\Z[\zeta_p]$ is generated by $1-\zeta_p^{j}$
where $j$ is any integer $\not \equiv 0 \pmod{p}$,
and thus the ratio $(1-\zeta_p^j)/(1-\zeta_p^k)$ is a unit
for any pair of integers $j$, $k \not \equiv 0 \pmod{p}$. Note that
\[
\lambda=(1+\zeta_p)(1+\zeta_p^{-1})=\frac{(1-\zeta_p^2)(1-\zeta_p^{-2})}{(1-\zeta_p)(1-\zeta_p^{-1})}, \qquad
\mu=-\zeta_p^{-1}(1+\zeta_p+\zeta_p^{2})=-\zeta_p^{-1} \frac{(1-\zeta_p^3)}{(1-\zeta_p)},
\]
showing that $\lambda$, $\mu$ are units. Hence $F=\Q(\zeta_p)^+$ is exceptional
for all $p \ge 5$. Note that $F$ is cyclic of degree $(p-1)/2$.
Recall that a \textbf{Sophie Germain prime} is a prime $\ell$ such that $p=2\ell+1$
is also prime. 
For any Sophie Germain prime $\ell \ge 5$, the number field $F=\Q(\zeta_p)^+$
with $p=2\ell+1$ is an exceptional cyclic field of degree $\ell$ and so belongs to $\cF$.
It is conjectured that there are infinitely many Sophie Germain
primes \cite[page 123]{Shoup}, and this conjecture would imply that $\cF$ 
is infinite. 

\section{Ramification in Cyclic Fields of Prime Degree}

\begin{lemma}\label{lem:totram}
Let $\ell$ be a prime. Let $F$ be a cyclic number field of degree $\ell$.
Write $\Delta_F$
for the discriminant of $\OO_F$. Let $p$ be a prime that ramifies
in $F$. Then the following hold.
\begin{enumerate}[(i)]
\item $p$ totally ramifies in $F$.
\item If $p \ne \ell$ then $\ord_p(\Delta_F)=\ell-1$.
\end{enumerate}
\end{lemma}
\begin{proof}
Let $I \subseteq \Gal(F/\Q)$ be an inertia subgroup for $p$.
Since $p$ ramifies, $I \ne 1$. As $\Gal(F/\Q)$ has prime order,
$I=\Gal(F/\Q)$. Hence $p$ is totally ramified in $F$,
and we can write $p\OO_F=\fp^\ell$ where $\fp$ is the unique
prime ideal above $p$.

We now prove (ii). Suppose $p \ne \ell$, therefore $p$
is tamely ramified in $F$.
Write $\mathfrak{D}_F$
for the different ideal for the extension $F/\Q$.
As the ramification degree is $\ell$,
we conclude \cite[page 199]{Neukirch} that $\ord_{\fp}(\mathfrak{D}_F)=\ell-1$.
However \cite[page 201]{Neukirch}, 
the discriminant and different are related by $\lvert \Delta_F \rvert=
\Norm_{F/\Q}(\mathfrak{D}_F)$. Hence $\ord_p(\Delta_F)=\ell-1$. This completes the
proof.
\end{proof}

\begin{lemma}\label{lem:precond}
Let $m$, $n$ be positive integers with $m \mid n$. Let $\ell$ be a prime
and let $F$ be a cyclic number field of degree $\ell$.
If $F \subseteq \Q(\zeta_n)$ and $\ell \nmid [\Q(\zeta_n) : \Q(\zeta_m)]$ then
$F \subseteq \Q(\zeta_m)$.
\end{lemma}
\begin{proof}
Suppse $F \subseteq \Q(\zeta_n)$ but $F \not \subseteq \Q(\zeta_m)$. As $F$ has prime degree $\ell$
we have $F \cap \Q(\zeta_m)=\Q$. Thus
$[ F \cdot \Q(\zeta_m) \, : \, \Q(\zeta_m)]=[F:\Q]=\ell$.
However, $\Q(\zeta_m) \subseteq F \cdot \Q(\zeta_m) \subseteq \Q(\zeta_n)$.
Therefore $\ell \mid [\Q(\zeta_n):\Q(\zeta_m)]$,
giving a contradiction.
\end{proof}
\begin{lemma}\label{lem:cond}
Let $\ell$ be a prime
and let $F$ be a cyclic number field of degree $\ell$. Suppose
$\ell \nmid \Delta_F$. Then the conductor of $F$
is squarefree, and divisible only by primes $p \equiv 1 \pmod{\ell}$.
\end{lemma}
\begin{proof}
Let $n$ be the conductor of $F$.
The primes that ramify in $F$ are precisely the primes
dividing the conductor \cite[Corollary VI.6.6]{Neukirch}.
As $\ell \nmid \Delta_F$
we see that $\ell \nmid n$. 

We would like to show that $n$ is squarefree. Suppose that $n$ is not
squarefree. Then we may write 
$n=p^r n^\prime$ where $p$ is a prime, $r \ge 2$, and $p \nmid n^\prime$.
Let $m=p n^\prime$.
We denote Euler's totient function by $\varphi$.
Then
\[
[\Q(\zeta_n) : \Q(\zeta_m)]=\frac{\varphi(n)}{\varphi(m)}=
\frac{(p-1) p^{r-1} \varphi(n^\prime)}{(p-1) \varphi(n^\prime)}=p^{r-1}.
\]
This is not divisible by $\ell$ and so by Lemma~\ref{lem:precond},
$F \subseteq \Q(\zeta_m)$. But $m<n$, contradicting the fact that $n$
is the conductor of $F$. It follows that $n$ is squarefree. 

Next let $p \mid n$ and write $n=pm$ with $p \nmid m$. Then
\[
[\Q(\zeta_n) : \Q(\zeta_m)]=p-1.
\]
By Lemma~\ref{lem:precond} and the definition of conductor we have
$\ell \mid (p-1)$.
\end{proof}

\section{The Unit Equation and Ramification}
We  now prove one of the claims in Theorem~\ref{thm:discs}.
\begin{lemma}\label{lem:Rlne0}
Let $\ell \ne 3$ be a prime.
Let $R_{\ell}$ be given by \eqref{eqn:Rl}, then $\ell \nmid R_\ell$.
In particular, $R_{\ell} \ne 0$.
\end{lemma}
\begin{proof}
Suppose $\ell \mid R_\ell$.
Then the polynomials $X^{2\ell}-1$
and $(X-1)^{2\ell}-1$ have a common root $\theta \in \overline{\F}_\ell$.
But in $\F_\ell[X]$ we have
\[
X^{2\ell}-1=(X^2-1)^\ell=(X-1)^\ell (X+1)^\ell,
\qquad
(X-1)^{2\ell}-1=\left((X-1)^2-1 \right)^\ell=X^\ell (X-2)^\ell.
\]
Hence $\theta \in \{1,-1\} \cap \{0,2\} \subset \F_{\ell}$.
As $\ell \ne 3$ this intersection is empty, giving a contradiction,
so $\ell \nmid R_{\ell}$.
\end{proof}
\noindent \textbf{Remark.} Lemma~\ref{lem:Rlne0} is false for $\ell=3$.
Indeed, $(1+\sqrt{-3})/2$ is a common root 
to $X^6-1$ and $(X-1)^6-1$, thus $R_3=0$.

\bigskip

For the remainder of this section $F$ will be a cyclic number field
of prime degree $\ell \ge 5$. By Lemma~\ref{lem:totram},
every rational prime $p$ which ramifies in $F$
is in fact totally ramified, and so there is 
a unique prime $\fp$ of $F$ above $p$. The prime $\fp$
must have inertial degree $1$, and so
$\OO_F/\fp \cong \F_p$.
\begin{lemma}\label{lem:totramification}
Let $\lambda \in \OO_F^\times$. Let $b \in \Z$
satisfy $\lambda \equiv b \pmod{\fp}$.
Then $b^\ell \equiv \pm 1 \pmod{p}$.
\end{lemma}
\begin{proof}
As $\fp$ is the unique prime
	above $p$ we have $\fp^\sigma=\fp$ for all $\sigma \in G=\Gal(F/\Q)$.
	Applying $\sigma$ to $\lambda \equiv b \pmod{\fp}$
	gives $\lambda^\sigma \equiv b \pmod{\fp}$. 
	Hence
\[
	\pm 1=\Norm_{F/\Q}(\lambda)=\prod_{\sigma \in G} \lambda^\sigma \equiv
	b^\ell \pmod{\fp}.
\]
Since $b^\ell$ is a rational integer,
$b^\ell \equiv \pm 1 \pmod{p}$.
\end{proof}

\begin{lemma}\label{lem:Rl}
Suppose the unit equation \eqref{eqn:unit}
has a solution. Let $R_\ell$ be as in \eqref{eqn:Rl}.
Then
every prime $p$ ramifying in $F$ satisfies $p \mid R_{\ell}$.
\end{lemma}
\begin{proof}
Let $(\lambda,\mu)$ be a solution to the unit equation.
Let $p$ be a prime ramifying in $F$ and let $\fp$
be the prime above it. 
Write $\lambda \equiv b \pmod{\fp}$ and $\mu \equiv c \pmod{\fp}$
with $b$, $c \in \Z$. By Lemma~\ref{lem:totramification},
$b^{2\ell} \equiv  1 \pmod{p}$ and $c^{2\ell} \equiv  1 \pmod{p}$.
However,
$\lambda+\mu=1$. Hence $c \equiv 1-b \pmod{p}$. 
Therefore $(b-1)^{2\ell}=(1-b)^{2\ell} \equiv c^{2\ell} \equiv 1 \pmod{p}$. Hence, 
the polynomials $X^{2\ell}-1$ and $(X-1)^{2\ell}-1$ have a common root
in $\F_p$, showing that $p \mid R_\ell$. 
\end{proof}

\section{Proof of Theorem~\ref{thm:discs}}
We now prove Theorem~\ref{thm:discs}. Thus let $F$
be a cyclic number field of degree $\ell \ge 5$ such that the unit
equation \eqref{eqn:unit} has solutions. 
Let $R_\ell$ be given by \eqref{eqn:Rl}. From Lemma~\ref{lem:Rlne0}
we know that $R_\ell \ne 0$.
Let $S_\ell$ be given by \eqref{eqn:Sl}.
We \textbf{claim} that every prime $p$ ramified in $F$ belong to $S_\ell$.
First note that every ramified $p$ divides $R_\ell$ by Lemma~\ref{lem:Rl}.
Next note that $\ell \nmid R_\ell$ by Lemma~\ref{lem:Rlne0}.
Thus $\ell$ is unramified in $F$, and so $\ell \nmid \Delta_F$.
Now Lemma~\ref{lem:cond} tells us that every ramified $p \equiv 1 \pmod{\ell}$.
This completes the proof of the claim. 

Let $T$ be the
set of primes dividing the discriminant $\Delta_F$.
This is also the set of primes dividing the conductor $N_F$
(see for example \cite[Corollary VI.6.6]{Neukirch}).
We know from the claim that $T$ is a subset of $S_\ell$.
Moreover, by a famous theorem of Minkowski \cite[Theorem III.2.17]{Neukirch} 
there are no number fields of discriminant $\pm 1$,
and thus $T \ne \emptyset$. 

Next, by part (ii) of Lemma~\ref{lem:totram},
and Lemma~\ref{lem:cond}
\begin{equation}\label{eqn:disc2}
\Delta_F= g \cdot \prod_{p \in T} p^{\ell-1}, \qquad
N_F=\prod_{p \in T} p,
\end{equation}
where $g=\pm 1$. 
However, as $F$ is Galois
of odd degree, it is totally real, and therefore the discriminant
is positive, so $g=1$. This completes the proof.

\section{Proof of Corollary~\ref{cor}}\label{sec:cor}
Let $F$ be an exceptional cyclic quintic field.
We apply Theorem~\ref{thm:discs} with $\ell=5$. Then
\[
R_5=\Res(X^{10}-1,(X-1)^{10}-1)=-210736858987743=-3 \times 11^9 \times 31^3.
\]
Thus $S_5=\{11,31\}$. We obtain
three possibilities for the conductor $N_F$: $11$, $31$, $341=11 \times 31$.
Thus $F$ is a degree $5$ subfield of $\Q(\zeta_{11})$, $\Q(\zeta_{31})$
or $\Q(\zeta_{341})$.
These respectively have Galois groups isomorphic to
$\Z/10\Z$, $\Z/30\Z$ and $\Z/10\Z \times \Z/30\Z$. By the Galois
correspondence,
$\Q(\zeta_{11})$ and $\Q(\zeta_{31})$ both have a unique subfield
of degree $5$, which we denote by $F_{11}=\Q(\zeta_{11})^+$ and $F_{31}$.
The group $\Z/10\Z \times \Z/30\Z$ has six subgroups of index $5$,
and so we obtain six subfields of $\Q(\zeta_{341})$ of degree $5$.
However, two of these are $F_{11}$ and $F_{31}$, so we only
obtain four new fields which we denote by
$F_{341,1}$, $F_{341,2}$, $F_{341,3}$, $F_{341,4}$.
We found defining polynomials for all these number fields
in the John Jones Number Field Database \cite{JR}, which we reproduce
in Table~\ref{table}.
\begin{table}
\begin{tabular}{||c|c||}
\hline
field & defining polynomial\\
\hline\hline
$F_{11}=\Q(\zeta_{11})^+$ & $x^5 - x^4 - 4x^3 + 3x^2 + 3x - 1$\\
\hline
$F_{31}$ & $x^5 - x^4 - 12x^3 + 21x^2 + x - 5$\\
\hline
$F_{341,1}$ & $x^5 + x^4 - 136x^3 - 300x^2 + 2016x + 3136$\\
\hline
$F_{341,2}$ & $x^5 + x^4 - 136x^3 + 41x^2 + 3039x + 1431$\\
\hline
$F_{341,3}$ & $x^5 + x^4 - 136x^3 + 723x^2 - 1053x + 67$\\
\hline
$F_{341,4}$ & $x^5 + x^4 - 136x^3 - 641x^2 - 371x + 67$\\
\hline
\end{tabular}
\caption{Cyclic number fields $F$ with conductor dividing $341=11 \times 31$.}
\label{table}
\end{table}

We used the unit equation solver in the computer algebra package
	\texttt{Magma} \cite{magma}. This is an implementation
	of the de Weger algorithm for solving unit equation
	with improvements due to Smart \cite{Smart}.
Applying the solver to our six number fields
we find that the unit equation \eqref{eqn:unit} does not
have solutions for $F=F_{31}$ and $F=F_{341,i}$ with $i=1,\dots,4$.
It does however have $570$ solutions for $F=F_{11}=\Q(\zeta_{11})^+$.

\end{document}